\documentclass[11pt]{amsart}
\input epsf
\usepackage{latexsym}
\usepackage{amssymb,amsmath,amsthm,amscd, amssymb,
graphicx, enumerate, verbatim, mathrsfs, color, url}
\usepackage[all]{xy}\usepackage[labelformat=empty]{caption}

\numberwithin{equation}{section}
\newtheorem{cor}[equation]{Corollary}
\newtheorem{lem}[equation]{Lemma}

\newtheorem{thm}[equation]{Theorem}

\newtheorem{quest}[equation]{Question}

\newtheorem{Example}[equation]{Example}
\newenvironment{ex}{\begin{Example}\rm}{\end{Example}}
\newtheorem{remark}[equation]{Remark}
\newenvironment{rmk}{\begin{remark}\rm}{\end{remark}}
\def\co{\colon\thinspace}

\newcommand{\epi}{\mbox{epi\,}}

\newcommand{\supp}{\mbox{supp\,}}

\newcommand{\Hess}{\mbox{Hess}}

\newcommand{\e}{\varepsilon}

\def\a{\alpha}
\def\de{\delta}

\def\g{\gamma}
\def\o{\omega}

\def\d{\partial}

\def\o{\omega}

\def\s{\sigma}
\def\l{\lambda}

\def\R{\mathbb{R}}
\def\N{\mathbb{Z}_{\ge 0}}
\def\Z{\mathbb{Z}}

\def\Z{\mathbb{Z}}
\def\S1{\bf S^1}

\newcommand{\C}{{\mathbb C}}

\makeatletter
\def\equalsfill{$\m@th\mathord=\mkern-7mu
\cleaders\hbox{$\!\mathord=\!$}\hfill
\mkern-7mu\mathord=$}
\makeatother

\begin{document}

\abovedisplayskip=6pt plus3pt minus3pt
\belowdisplayskip=6pt plus3pt minus3pt

\title[Smoothness of Minkowski sum and generic rotations]
{\bf Smoothness of Minkowski sum and generic rotations}

\thanks{
%52A10 View Publications (1973-now) Convex sets in 2 dimensions (including convex curves)
%52A41 View Publications (1991-now) Convex functions and convex programs
}\rm

\author{Igor Belegradek\and Zixin Jiang}

\address{Igor Belegradek\\School of Mathematics\\ Georgia Institute of
Technology\\ Atlanta, GA, USA 30332-0160}\email{ib@math.gatech.edu}

\address{Zixin Jiang\\Department of Mathematics\\
University of California, Berkeley\\
970 Evans Hall, Berkeley, CA 94720-3840}\email{zixin\_jiang@berkeley.edu}

%\thanks{}

\date{}
\begin{abstract}
Can the Minkowski sum of two convex bodies be made smoother by rotating one of them? 
We construct two $C^\infty$ strictly convex plane bodies such that
after any generic rotation (in the Baire category sense) of one of the summands the  Minkowski sum is not $C^5$.
On the other hand, if for one of the bodies the zero set of the Gaussian curvature 
has countable spherical image, we show that any generic rotation  
makes their Minkowski sum as smooth as the summands. 
We also improve and clarify some previous results on smoothness of the Minkowski sum.
\end{abstract}
\maketitle
%\tableofcontents

\section{Introduction}

A {\em convex body\,} in $\R^n$ is a compact convex set with non-empty interior.
A {\em plane convex body\,} is a convex body in $\R^2$.
A convex body is {\it $C^{k,\a}$\,} if its boundary is a $C^{k,\a}$ submanifold of $\R^n$.
Here $0\le \a\le 1$ and $k$ is a nonnegative integer, $\infty$ or $\o$. As usual 
$C^\o$ means analytic, while $C^{k,\a}$ refers to functions whose $k$th partial derivatives
are continuous if $\a=0$, Lipschitz if $\a=1$, and $\a$-H\"older continuous if $0<\a<1$.
We follow the convention $\o>\infty>l$ for any $l\in \Z$. 
The {\em Minkowski sum\,} $A+B=\{a+b:a\in A, b\in B\}$ of two convex bodies $A$, $B$ in $\R^n$ is 
a convex body in $\R^n$. 

S.~Krantz and H.~Parks showed in~\cite{KraPar} that if $A$ is $C^{1,\a}$, then so is $A+B$. 
In dimensions $n\ge 3$ there are examples where $A$, $B$ are $C^\infty$ but $A+B$ is not $C^2$, 
see~\cite[Theorem 3.4]{Kis-shad} for $n=3$ and~\cite{Bom-analy} for $n\ge 4$. 
For $n=2$ one has: \ \vspace{-2pt}
\begin{itemize}
\item 
If $A$, $B$ are $C^\o$, then $A+B$ is $C^{6,\frac{2}{3}}$
and this is sharp~\cite{Kis, Kis-regul}. \vspace{3pt}
\item
If $A$, $B$ are $C^k$ with $k\in \{1,2,3,4\}$, then $A+B$ is $C^k$, and this is sharp
in the sense that for any $\a\in (0,1)$ there are $C^\infty$ plane convex bodies $A$, $B$ 
such that $A+B$ is not $C^{4,\a}$~\cite{Bom-sum}. 
\end{itemize}

For a $C^2$ convex body $C$ in $\R^n$ let 
$\nu_{_{\!C}}\co\d C\to S^{n-1}$ denote the {\em Gauss map\,} given by the unit outer normal, let $G_{_{\!C}}$ 
be the corresponding {\em Gaussian curvature}, i.e. the product of the eigenvalues of the differential of $\nu_{_{\!C}}$.

Let us start from a basic result on smoothness of the Minkowski sum that improves on some 
work in~\cite{Bom-analy, KraPar, Gho} discussed in Remark~\ref{rmk: hist on smoothness}.

\begin{thm}
\label{thm: a+b C^k} 
%Let $\a\in [0,1]$ and $k\ge 2$ be an integer, $\infty$ or $\o$.
Let $A$, $B$ be $C^2$ convex bodies in $\R^n$, and let $a$, $b$, $a+b$ be
boundary points of $A$, $B$, $A+B$, respectively.
If $A$ is $C^{k,\a}$ near $a$ and $G_{_{\!A}}(a)\neq 0$, then \vspace{-3pt}
\begin{itemize}
\item[\textup(i)] 
$A+B$ is $C^2$ near $a+b$,\vspace{3pt}
\item[\textup(ii)]
$G_{_{\!B}}(b)\neq 0$ %\ \Leftrightarrow\  
if and only if 
$G_{_{\!A+B}}(a+b)\neq 0$, \vspace{3pt}
\item[\textup(iii)]  
$B$ is $C^{k,\a}$ near $b$ %\ \Leftrightarrow\   
if and only if 
$A+B$ is $C^{k,\a}$ near $a+b$.
\end{itemize}
\end{thm}

For example, if $A$ is a round ball
and $B$ is any $C^2$ convex body of the same dimension, then
the boundaries of $B$ and $A+B$ are parallel 
convex hypersurfaces and according to Theorem~\ref{thm: a+b C^k} 
if one of them is $C^{k,\a}$, so is the other one. 
More generally, if $G_{_{\!A}}$ is nowhere zero
and either $A$ is $C^\o$ or $B$ is less smooth than $A$, then
Theorem~\ref{thm: a+b C^k} gives the optimal regularity for $A+B$, 
which is the largest $k$, $\a$ such that both $A$ and $B$ are $C^{k,\a}$. 

If $B$ is at least as smooth as $A$, then the regularity of $A+B$ predicted by
Theorem~\ref{thm: a+b C^k} need not be optimal, and indeed,
$A+B$ can be smoother at $a+b$ than $A$ at $a$.
Let us sketch two different ways of how this could happen.
In Example~\ref{ex: sum of epigraphs} we describe a family of Minkowski sums $A+B$
such that $A$, $B$ are $C^\o$ and they can be can be cut and rearranged resulting in the Minkowski sum 
$A^\prime+B^\prime=A+B$ where $A^\prime$, $B^\prime$ are merely $C^{k}$
for any given $k\ge 2$, or even $C^{1,1}$.
In Example~\ref{ex: flat face} we note that if $b$ lies in the interior of a flat face $F$ of $B$,
then $a+F$ is a face of $A+B$, which makes
$A+B$ analytic near $a+b$, no matter how (non)smooth $A$ is at $a$.

This paper was motivated by the following question of M.~Ghomi:  
{\em Can the Minkowski sum of convex bodies be made smoother by rotating one of the summands?\,} 
More precisely, if $A$, $B$ are $C^{k,\a}$, we seek rotations $R$ such that
$R(A)+B$ is $C^{k,\a}$.
 
Recall that the boundary points $a$, $b$, $a+b$ have the same spherical image.
Hence in view of Theorem~\ref{thm: a+b C^k} it is relevant to consider the set
$Z_{_{\!C}}=\{\nu_{_{\!C}}(x) : G_{_{\!C}}(x)=0\}$, the spherical image
of the zero set of the Gaussian curvature.
%the set of points on $\d C$ where $G_{_{\!C}}$ vanishes. 
Clearly any rotation $R\in SO(n)$ maps $Z_{_{\!C}}$ to $Z_{_{\!R(C)}}$. 
%because $\nu_{_{\!R(C)}}=R\circ \nu_{_{\!C}}\circ R^{-1}$ and $G_{_{\!R(C)}}\circ R=G_{_{\!C}}$.
Now a natural question is: Can one always rotate $Z_{_{\!A}}$ off $Z_{_{\!B}}$?
For such rotations $R(A)+B$ will be $C^{k,\a}$.
To this end we prove:

\begin{thm}
\label{thm: countable zero set}
Let $A$, $B$ be $C^{k,\a}$ convex bodies in $\R^n$
where $k\ge 2$ is an integer, $\infty$ or $\o$, and $\a\in [0,1]$. If either 
$Z_{_{\!A}}$ or $Z_{_{\!B}}$ is countable, then  
$R(A)+B$ is $C^{k,\a}$ for any generic $R\in SO(n)$.
\end{thm}

Recall that a rotation in $SO(n)$ with a certain property is {\em generic\,}
if the set of rotations with the property is {\em comeager}, i.e., 
a countable intersection of open dense subsets of $SO(n)$. Thus the conclusion of
Theorem~\ref{thm: countable zero set} is that the set of rotations
$R$ such that $R(A)+B$ is $C^{k,\a}$ is comeager in $SO(n)$.
Standard Baire category considerations, see e.g.~\cite[Sections 8A, 8B]{Kec-book}, 
show that if $S$ is a comeager subset of $SO(n)$, then
$SO(n)\setminus S$ is not comeager (even though 
it could be uncountable and dense), and furthermore,
$S$ contains an uncountable dense $G_\de$ subset of $SO(n)$.

\begin{proof}[Proof of Theorem~\ref{thm: countable zero set}]
Set $R_{A,B}=\{R\in SO(n): R(Z_{_{\!A}})\cap Z_{_{\!B}}=\emptyset\}$.
By Theorem~\ref{thm: a+b C^k} above $R(A)+B$ is $C^{k,\a}$ for every $R\in R_{A,B}$.
We are going to show that $R_{A,B}$ is a dense $G_\delta$ subset of $SO(n)$.
Let $C\in\{A,B\}$. Since the Gaussian curvature is continuous,
$Z_C$ is closed. Also $Z_C$ is precisely the set of singular values
for the Gauss map of $C$, so by Sard's theorem
$Z_C$ has Lebesgue measure zero, and hence it is nowhere dense.
(Sard's theorem applies to $C^1$ maps,
and the Gauss map is $C^1$ since $C$ is $C^2$). 
For any $x\in S^{n-1}$ the map $p_x\co SO(n)\to S^{n-1}$ given by $p_x(R)=R(x)$
is a submersion, hence the $p_x$-preimage of any closed nowhere dense
subset is closed and nowhere dense. Since
$p_x^{-1}(Z_B)=\{R\in SO(n)\,|\, R(x)\in Z_B\}$,
we conclude that $R(Z_A)\cap Z_B\neq\emptyset$ if and only if 
$R\in\bigcup_{_{x\in Z_A}} p_x^{-1}(Z_B)$, which is a union of nowhere dense closed sets 
indexed by $Z_A$, and which is precisely $SO(n)\setminus R_{A,B}$.
Now if $Z_A$ is countable, then by the Baire category theorem any countable union of nowhere dense
sets has empty interior. Hence $R_{A,B}$ is dense. Since $SO(n)\setminus R_{A,B}$ is a countable union of closed sets,
$R_{A,B}$ is $G_\de$, i.e., a countable intersection of open sets.  
The case when $Z_B$ is countable follows by observing that
the self-map of $SO(n)$ taking an element to its inverse maps $R_{A,B}$
to $R_{B,A}$, and like any homeomorphism preserves the property of being dense and $G_\delta$.
\end{proof}

\begin{cor}
\label{cor: anal}
If $A$, $B$ are $C^2$ plane convex bodies that are $C^\o$ away from countable subsets of $\d A$, $\d B$, 
then $R(A)+B$ is $C^\o$ for any generic $R\in SO(2)$.
\end{cor}
\begin{proof}[Proof of Corollary~\ref{cor: anal}]
Let $D\in\{A, B\}$. The complement of a countable subset of $\d D$
is a real analytic curve and (like any metrizable manifold)
it has at most countably many components. Fix one such component $L$ and 
let $x\in L$. Then near $x$ the curve $L$ can be written
as the graph of a real analytic function $f$ over the tangent line at $x$
%\cite[Proposition 2.7.3]{KraPar-book}
so that $f^{\prime\prime}$ is also real analytic, and hence
its zeros are isolated.
% \cite[Proposition 1.2.6]{KraPar-book}
In these coordinates the curvature of $\d D$ equals
$\frac
{|f^{\prime\prime}|}
{(1+(f^\prime)^2)^{3/2}}$, so the zeros
of the curvature are also isolated. Since $L$ is the union of countably many
compact sets, the curvature of $L$ has at most countably many zeros, 
so Theorem~\ref{thm: countable zero set} applies. 
\end{proof}

The main result of this paper is that the analog of Corollary~\ref{cor: anal} for $C^\infty$ plane
convex bodies fails:

\begin{thm}
\label{intro: main}
For any $\a\in (0,1)$ there are $C^\infty$ strictly convex plane convex bodies $A$, $B$ such that 
$R(A)+B$ is not $C^{4,\a}$ for any generic $R\in SO(2)$.
\end{thm}

There is substantial flexibility in the construction of $A$, $B$ in Theorem~\ref{intro: main}
and in particular, for any positive integer $m$ 
we can arrange that $B=A$ where $A$ is invariant under the rotation by the angle 
$\frac{2\pi}{m}$. Thus we get an order $m$ subgroup of $SO(2)$
such that $R(A)+B=2A$ is $C^\infty$ for any $R$ in the subgroup. 
This explains why our method of proving Theorem~\ref{intro: main} omits certain rotations, 
and calls for the following question:

\begin{quest}
Do there exist $C^\infty$ convex bodies $A$, $B$ in $\R^n$ such that $R(A)+B$ is not $C^\infty$
for every $R\in SO(n)$?
\end{quest}

Let us sketch the proof of Theorem~\ref{intro: main}. We make extensive use
of the work of J.~Boman in~\cite{Bom-sum} who gave a proof for $R=1$ by constructing two
$C^\infty$ strictly convex functions $f$, $g$ on $[0,\infty)$
such that the Minkowski sum of their epigraphs $\{(x,y)\in \R^2: f(x)\le y\}$,
$\{(x,y)\in \R^2: g(x)\le y\}$ is not $C^{4,\a}$ at $0$, where all derivatives
of $f$, $g$ vanish. The Minkowski sum is the epigraph of the infimal convolution $f\Box g$,
and the $C^{4,\a}$ norm $f\Box g$ blows up near $0$.

By the above discussion
a necessary condition is that $Z_{_{\!A}}$, $Z_{_{\!B}}$ are 
uncountable closed subsets of $S^1$ that
cannot be rotated off each other. Moreover, 
Sard's theorem implies that $Z_{_{\!A}}$, $Z_{_{\!B}}$ have measure zero,
and in particular they are totally disconnected. Recall that any 
compact totally disconnected space without isolated points is homeomorphic to
a Cantor set. 
Pulling these sets to the fundamental domain $[0, 2\pi]$ in the universal cover of $S^1$
we need to know when the difference of two totally disconnected subsets 
of $[0, 2\pi]$ equals $[0,2\pi]$ mod $2\pi$. Differences of Cantor-like sets were studied,
and there are some sufficient conditions ensuring that the difference fill the whole interval.
Of course, the mere fact that $Z_{_{\!A}}$, $Z_{_{\!B}}$ 
cannot be rotated off each other is not enough for non-smoothness.

Strict convexity implies that the Gauss maps are homeomorphisms, hence we wish to
build a strictly convex curve whose curvature vanishes on a Cantor-like set.
The idea is to cook up a strictly convex function $\hat f$ on $[0,\tau]$ 
such that near the endpoints $\hat f$ equals $f$ and $f\circ l$
where $l(x)=\tau-x$. The construction should be such that we 
control all derivatives of $\hat f$ as $\tau\to 0$, which is surprisingly delicate.  
Then we piece together the graphs of such functions $\hat f$ for various $\tau$ in a Cantor
set like pattern, and with the above derivative control this curve becomes
the boundary of a $C^\infty$ strictly convex set $A$. Similarly, we use $g$
to construct a convex set $B$. 
 
There are two kinds of points in a Cantor-like set: the endpoints of removed open intervals
and the limits of the endpoints. The set of endpoints is countable. 
If the rotation angle is the difference of two (say left) endpoints, it superimposes
graphs of $\hat f$ and $\hat g$, and we are done by Boman's work.
Such rotation angles form a countable dense subset of $S^1$.
If the rotation angle is not in countable dense subset, we have to investigate how the 
$C^{4,\a}$ norm of $f\Box g$ changes under small rotation of the graph of $f$.
As Boman shows in~\cite{Bom-sum} the infimal convolution $f\Box g$ does not depend continuously of $f$, $g$
in the $C^4$ topology (let alone in any stronger topology). Nevertheless our deformation
of $f$ is very specific and we prove continuity under such deformations, and again reduce
to Boman's computations to show that nearby rotations lead to Minkowski sums with large $C^{4,\a}$.
This implies that there is a dense $G_\de$ set of rotations such that the corresponding
Minkowski sums are not $C^{4,\a}$.

{\bf Structure of the paper:} 
In Section~\ref{sec: inf convolution} we prove 
Theorem~\ref{thm: a+b C^k}.
Sections~\ref{sec: rotated graphs}--\ref{sec: smoothing a hinge}
contain technical results needed to establish Lemma~\ref{lem: smooth hinge}, which is
used in Section~\ref{sec: building a curve} to construct $A$, $B$ appearing in Theorem~\ref{intro: main},
which is proved in Section~\ref{sec: proof of main thm}.
Another technical ingredient is a proof of~\cite[Lemma 2]{Bom-sum} given in the appendix
for completeness.

\section{On smoothness of infimal convolution}
\label{sec: inf convolution}

The purpose of this section is to provide a reference for 
basic smoothness properties of the infimal convolution, which do not seem to appear
in the literature. The previous work was reviewed in the introduction and Remark~\ref{rmk: hist on smoothness}.

Let $f$, $g$ be real-valued convex functions on $\R^n$ 
that are bounded below. 
Their {\em infimal convolution\,} $h=f\Box g$ is given by 
\begin{equation}
\label{form: defn inf conv}
\displaystyle{h(x)=\inf_{y}(f(y)+g(x-y))=\inf_{x_1+x_2=x}(f(x_1)+g(x_2))}.
\end{equation}
Note that $h$ is real-valued and bounded below.
The right hand side of (\ref{form: defn inf conv}) gives $f\Box g=g\Box f$.
The epigraph of $h$ is the Minkowski sum of epigraphs of $f$ and $g$, 
see e.g.~\cite[p.39]{Sch-book}, hence $h$ is convex.

We will write $D>0$ to mean that the matrix $D$ is positive definite.

\begin{lem} 
\label{lem: min, IFT}
Let $f$, $g$ be convex functions from $\R^n$ to $[0,\infty)$
such that $f$, $g$ are $C^2$ near $0$ and $f(0)=0=g(0)$. 
Set  $h=f\Box g$ and $\s_x(y)=f(y)+g(x-y)$. Suppose that 
$\,\mathrm{Hess}\, \s_0\vert_{_{y=0}}>0$. Then \vspace{-2pt}
\begin{itemize}
\item[\textup(a)]
There is a neighborhood $U$ of $0$ in $\R^n$
and a $C^1$ map $\mu\co U\to\R^n$ such that $\mu(0)=0$
and $\s_x(y)>\s_x(\mu(x))$ for each $x\in U$ and every $y\neq \mu(x)$. \vspace{3pt}
\item[\textup(b)]
If in addition $f$, $g$ are $C^{k,\a}$,
then $\mu$ is $C^{k-1,\a}$ and $h$ is $C^{k,\a}$.\vspace{3pt}
\item[\textup(c)]
$h$ is $C^2$ and satisfies the formulas \textup{(\ref{form: h via y(x)}), (\ref{form: hess with mu})} below.\vspace{3pt}
\item[\textup(d)]
If $\mathrm{Hess}\, h\vert_{_{x=0}}>0$, then 
$\mathrm{Hess}\, f$, $\mathrm{Hess}\, g$ are positive definite near $0$, and
the matrices $J_\mu$, $I-J_\mu$ defined after the formula \textup{(\ref{form: h via y(x)})} below 
are non-singular. 
\end{itemize}
\end{lem}
\begin{proof}
Since $f$ and $g$ are convex, so is $\s_x$
(because convexity is inherited under affine
change of variable, and under addition). 
Fix any closed round ball $B$ about $0$ in $\R^n$
such that $\mathrm{Hess}\, \s_0\vert_{_B}>0$. Then
$\mathrm{Hess}\,\s_x\vert_{_B}>0$ for all $x$ near $0$.
Thus $\s_x\vert_{_B}$ is strictly convex, and hence has a unique minimum,
which we denote $\mu_{_B}(x)$.
The assumptions imply that $0$ is a unique minimum of $\s_0$,
hence there is a constant $r_{_B}>0$ such that $\s_x\vert_{_{\d B}}>r_{_B}$
for all $x$ near $0$.
Since $\s_x(0)=g(x)$ tends to $0$ as $x\to 0$,
we can assume that $x$ is so small that
$\s_x(0)<r_{_B}$. Convexity of $\s_x$ ensures that $\s_x> r_{_B}$ outside $B$.
(Otherwise, given $y\notin B$ with $\s_x(y)\le r_{_B}$ 
we could find $t\in [0,1]$ with $ty\in\d B$
which contradicts $\s_x(ty+(1-t)0)\le t\s_x(y)+(1-t)\s_x(0)<r_{_B}$.)
Hence $\mu_{_B}(x)$ is a unique minimum of $\s_x$ for every $x$ near $0$. 
In particular, since $0$ is a minimum of $\s_0$ we get
$\mu_{_B}(0)=0$.

Critical points of $\s_x$ are solutions of
\begin{equation}
\label{form: crit point}
0=\nabla \s_x\vert_{y}=\nabla f\vert_{y}-\nabla g\vert_{x-y}.
\end{equation} 
Since $0$ is a minimum of $f$, $g$, the right hand side of (\ref{form: crit point}) vanishes when $x=0=y$, and its
Jacobian with respect to $y$ equals $\mathrm{Hess} f\vert_{y}+\mathrm{Hess}\, g\vert_{x-y}$ which
is positive definite if $y=0$. By the Implicit Function Theorem
(\ref{form: crit point}) has a unique $C^1$ solution $y=\mu(x)$ satisfying $\mu(0)=0$
and defined near $0$; moreover, if $f$, $g$ are $C^{k,\a}$, then $\mu$ is $C^{k-1,\a}$.
(The $C^{k,\a}$ version of the Implicit Function Theorem
follows in the usual way from the  $C^{k,\a}$ version of the Inverse
Function Theorem proved in~\cite[Section 2.2]{BHS}). 
By uniqueness $\mu=\mu_{_B}$ whenever both sides are defined, which is a neighborhood of $0$.
This completes the proof of (a).

For $x\in U$ we have  $h(x)=f(\mu(x)) +g(x-\mu(x))$ and hence
\begin{equation}
\label{form: h via y(x)}
\mbox{\small$
\nabla h\vert_x=\nabla f\vert_{\mu(x)}\cdot J_\mu\vert_x+\nabla g\vert_{x-\mu(x)}\cdot (I-J_\mu)\vert_x=\nabla g\vert_{x-\mu(x)}=
\nabla f\vert_{\mu(x)},
$}
\end{equation}
where we write gradients as row vectors, 
$\cdot$ is the matrix multiplication, $I$ is the identity map,
$J_\mu$ is the Jacobian matrix of $\mu$,
and the last two equalities follow from (\ref{form: crit point}).

If in addition $f$, $g$ are $C^{k,\a}$ near $0$, then
$\nabla f$ and $\mu$ are $C^{k-1, \a}$, and hence so is $\nabla h=\nabla f\circ \mu$, as the composite 
of $C^{k-1, \a}$ maps is $C^{k-1, \a}$, see~\cite[Section 2.2]{BHS}; thus $h$ is $C^{k, \a}$.

Since $f$, $g$ are assumed $C^2$, we conclude that $h$ is $C^2$.
Differentiating (\ref{form: h via y(x)}) gives 
\begin{equation}
\label{form: hess with mu}
\mathrm{Hess}\, h\vert_x=\mathrm{Hess}\, g\vert_{x-\mu(x)}\cdot (I-J_{\mu})\vert_x=\mathrm{Hess}\, f\vert_{\mu(x)}\cdot J_\mu\vert_x
\end{equation} 
In particular, if $\mathrm{Hess}\, h\vert_{_{x=0}}>0$, then near $0$
the matrices $J_\mu$, $I-J_\mu$ are non-singular and $\mathrm{Hess}\, f$, $\mathrm{Hess}\, g$, $\mathrm{Hess}\, h$ are positive definite. 
\end{proof}

\begin{rmk}
The assumption $\,\mathrm{Hess}\, \s_0\vert_{_{y=0}}>0$ of Lemma~\ref{lem: min, IFT} holds if 
$\Hess f\vert_{_{y=0}}>0$ or $\Hess\, g\vert_{_{y=0}}>0$.
\end{rmk}

\begin{lem}
\label{lem: inf conv equiv}
Let $f$, $g$, $h$, $\mu$, $U$ be as in Lemma~\ref{lem: min, IFT}. 
If $\,\mathrm{Hess}\, f\vert_{_{x=0}}>0$, then over $U$ the matrix
$I-J_\mu$ is non-singular, and the following equivalences hold: 
\[
\mathrm{Hess}\, h>0\ \Leftrightarrow\ \mathrm{Hess}\, g>0 \ \Leftrightarrow\ J_\mu\ \text{is\ non-singular},
\]
If in addition $f$, $h$ are $C^{k,\a}$ near $0$, then so is $g$.
\end{lem}
\begin{proof}
The map $\nabla f$ is a diffeomorphism near $0$ as $\Hess f>0$. Hence (\ref{form: h via y(x)}) gives
$\mu=(\nabla f)^{-1}\circ\nabla h$. Thus if $f$, $h$ are $C^{k,\a}$ near $0$, then
$\nabla f$ and $\mu$ are $C^{k-1, \a}$, and hence so is $\nabla g=\nabla f\circ \mu\circ (I-\mu)^{-1}$,
i.e., $g$ is $C^{k, \a}$. 

To show that $I-J_\mu$ is non-singular suppose $J_\mu(v)=v$ and use (\ref{form: hess with mu}) to compute \[
0=\mathrm{Hess}\, f\vert_{_{\mu(x)}} J_{\mu}\vert_{_x} v=\mathrm{Hess}\, f\vert_{_{\mu(x)}}  v,\] 
and so $v=0$ as $\mathrm{Hess}\, f>0$. Thus $I-\mu$ is a diffeomorphism near $0$,
and (\ref{form: hess with mu}) yields the desired equivalences.
\end{proof}

\begin{rmk}
\label{rmk: f non-smooth while g, h are}
If $f$, $g$, $h$, $\mu$, $U$ be as in Lemma~\ref{lem: min, IFT}, and in addition, $g$, $h$ 
are $C^{k,\a}$, then $f$ need not be $C^{k,\a}$. For example, if $g$ is everywhere zero,
then so is $h$, regardless of how smooth $f$ is, and we can arrange the graph of $f$
to have positive curvature everywhere.
%
%e.g. $f(x)=x^2+x^{2+\b}$ for $\b\in (0,1)$ is not in $C^{2,\a}$ for $\a\in (\b,1]$.
%
\end{rmk}

\begin{ex}
\label{ex: sum of epigraphs}
Given a function $q\co\R\to\R$,
let $q_-$, $q_+$, $\epi q$ denote the restrictions of $q$ to $(-\infty, 0]$, $[0,\infty)$,
and the epigraph of $q$, respectively.
If $q$ is convex with minimum at $0$, then it is easy to see that
$\epi q=\epi q_- + \epi q_+$. Let $r\co\R\to\R$ be another  
convex function with minimum at $0$, and let $f(q,r)\co\R\to\R$ be the function
whose restrictions to $(-\infty, 0]$, $[0,\infty)$ are $q_-$, $r_+$, respectively.
Note that $f(q,r)$ is also convex with minimum at $0$. Thus
\begin{equation*}
\epi f(q,r)+\epi f(r,q)=\epi q_- + \epi r_+ + \epi r_- + \epi q_+=
\epi q + \epi r,
\end{equation*}
or equivalently, $f(q,r)\Box f(r,q)=q\Box r$. 
For every $k\ge 2$ it is easy to find examples such that 
$q$, $r$ are $C^{k}$ near $0$ and $q^{\prime\prime}>0$, so that
$q\Box r$ is $C^{k}$ by Lemma~\ref{lem: min, IFT},
but $f(q,r)$, $f(r,q)$ are not $C^{k}$ (even though they are always $C^1$).
Also if $q(x)=x^2$ and $r(x)=ax^2$ where $0<a\neq 1$ is a constant, then
$f(q,r)$, $f(r,q)$ are $C^{1,1}$, convex, and not $C^2$ at $0$.
This is a rare case when the infimal convolution
can be explicitly computed and $f(q,r)\Box g(r,q)(x)=q\Box r=\frac{a}{a+1}x^2$.
\end{ex}

\begin{proof}[Proof of Theorem~\ref{thm: a+b C^k}] 
The smoothness of $A+B$ at $a+b$ is preserved under translations, so
replacing $A$, $B$ by $A-a$, $B-b$ we can assume $a=0=b$.
Since $A$, $B$ are $C^1$, so is $A+B$, see~\cite[Theorem 2]{KraPar}. 
A basic property of the Minkowski sum is that $\nu_{_{\!A}}(0)=\nu_{_{\!B}}(0)=\nu_{_{\!A+B}}(0)$,
see~\cite[p.351]{KraPar}. 

Let $\tau$ be the tangent plane to $A+B$ at $0$ with the orthogonal
complement $\tau^\bot$, which we linearly identify with $\R$ by mapping  $-\nu_A(0)$ to $1$. 
Then near $0$ the bodies $A$, $B$, $A+B$ coincide with epigraphs
of certain convex functions $f$, $g$, $h=f\Box g$ from $\tau$ to $\tau^\bot$ whose infima
are attained at $0$ (and possibly at other points). 

One way to construct $f$
is to consider a cone $C(\e_f)$ with apex at $0$ consisting of vectors that make the angle $\e_f>0$
with $-\nu_{_{\!A}}(0)$, where $\e_f$ is so small that $C(\e_f)$ lies in $A$ near $0$,
and then let $f$ be the function whose epigraph equals $A+C(\e_f)$; one defines $C(\e_g)$ similarly,
and then lets $h$ be the the function with epigraph $A+B+C(\e_f)+C(\e_g)$.

The function $f$, $g$, $h$ is $C^{k,\a}$ near a point if and only if so is
$A$, $B$, $A+B$, respectively, and in particular, $f$, $g$ are $C^2$ near $0$.
The Gaussian curvature is intrinsic up to sign, hence non-vanishing of 
$G_{_{\!A}}$, $G_{_{\!B}}$, $G_{_{\!A+B}}$ at a point is equivalent non-vanishing 
of all principal curvatures of the graph of $f$, $g$, $h$, respectively. 
The principal curvatures of the graph of a convex function
(with respect to the outer normal) are nonnegative, and equal 
the eigenvalues of $\Hess\, f$, see e.g.~\cite[p.115 and Theorem 1.5.13]{Sch-book}. 
Thus $\Hess f>0$ at $0$, and
the claimed result follows from Lemmas~\ref{lem: min, IFT} and \ref{lem: inf conv equiv}.
\end{proof}

\begin{ex} 
\label{ex: flat face}
Under the assumptions of Theorem~\ref{thm: a+b C^k}
suppose $a=0=b$ and let $H$ be the common supporting hyperplane at $0$
to $A$, $B$, $A+B$. If $\d B=H$ near $0$, then as in Remark~\ref{rmk: f non-smooth while g, h are}
one sees that the boundary of $A+B$ equals $H$ near $0$ regardless of
what $A$ is. Then $B$ and $A+B$ are $C^\o$ near $0$, while
we can specify the smoothness of $A$ at $a$ arbitrarily, and in particular,
choose it to be no better than $C^{k,\a}$. Note however that
the low regularity of $A$ at $a$ might still affect the smoothness of 
$A+B$ on the boundary of the face $H\cap (A+B)$.
\end{ex}

\begin{rmk} 
\label{rmk: hist on smoothness}
Lemma~\ref{lem: min, IFT}(c) is mentioned in~\cite[p.226]{Bom-analy}, and a closely related result
is discussed in~\cite[p.349, Remark 1]{KraPar}.
In~\cite[Proposition 5.1]{Gho} one finds a version of Lemma~\ref{lem: min, IFT}(b)
which treats the case $\a=0$ under the assumptions that $G_{_{\!A}}(a)\neq 0$ and
$B$ is strictly convex at $b$ (the latter assumption does not appear
in the  statement of~\cite[Proposition 5.1]{Gho} but is used in the proof: 
without it the set $K$ on the first line of~\cite[p.1605]{Gho} is not always 
strictly convex). 
\end{rmk}

\section{A bound on the $C^r$ norm of rotated graphs}
\label{sec: rotated graphs}

In this section we discuss rotated graphs, which feature prominently 
in our constructions.
Let $J\subset\R$ be a compact interval.
If $f\in C^\infty(J)$ and $\phi\in\R$ is sufficiently close to $0$,
then the rotation $z\to ze^{i\phi}$ takes
the graph of $f$ to the graph of a $C^\infty$ function, 
which we denote $f_\phi$. Thus $e^{i\phi}(x+if(x))=y+if_\phi(y)$. 
The real and imaginary parts of $e^{i\phi}(x+if(x))$ are 
$R(x)=x\cos\phi-f(x)\sin\phi$ and $I(x)=x\sin\phi+f(x)\cos\phi$, respectively.
Thus $f_\phi\circ R=I$, the domain of $f_\phi$ is $R(J)$, and we compute
\begin{equation}
\label{form: 1st 2nd deriv rotation}
f_\phi^\prime\circ R=\frac{I^\prime}{R^\prime}=
\frac{\sin\phi+f^\prime\cos\phi}{\cos\phi-f^\prime\sin\phi}
\quad\quad\quad\quad
f_\phi^{\prime\prime}\circ R=
\frac{I^{\prime\prime} R^\prime-I^\prime R^{\prime\prime}}
{(R^\prime)^3}=
\frac{f^{\prime\prime}}{(R^\prime)^3} 
\end{equation}
where $R^\prime$ does not vanish (because then $I^\prime$ would have to vanish
at the same point and $R^\prime$, $I^\prime$ clearly have no common zeros).

Given $r\in\N$ consider the norm
on $C^\infty(J)$ given by $\displaystyle{\|f\|_r=\sum_{i=0}^r\max |f^{(i)}|}$, which
clearly satisfies $\|fg\|_r\le \|f\|_r\|g\|_r$. When we need to specify the interval 
over which the maxima are taken we write $\|f\|_{_{C^r(J)}}$ for $\|f\|_r$.

\begin{lem}
\label{lem: rotation} 
Let $f\in C^\infty(J)$ and let $U$ be a compact subset of $(-\pi/2, \pi/2)$ 
such that $f_\phi$ is defined  and 
$\|f\tan\hspace{-1pt}\phi\|_r<1$ for each $\phi\in U$.
Then $\|f_\phi\|_{_{C^r(R(J))}}$ is bounded by a universal continuous
function of $\phi$, $r$, $\|f\|_{2r-1}$,
and $D=\mathrm{diam}(\{0\}\cup J)$.
\end{lem}
\begin{proof} Since
$\displaystyle{\max_{R(J)}|f_\phi^{(i)}|=\max_{J}|f_\phi^{(i)}\circ R|}$,
it is enough to estimate $|f_\phi^{(i)}\circ R|$. Note that
\[
\|f_\phi\|_{0}=\max_{x\in J}|I|= 
\max_{J} |x\sin\phi+f(x)\cos\phi|\le D+\|f\|_{0}.
\]
If $i\ge 0$, then
\[
\mbox{\small$
\|f^{(i+1)}_\phi\circ R\|_r=
\left\|\frac{(f^{(i)}_\phi\circ R)^\prime}{R^\prime}\right\|_r\le 
\left\|(f^{(i)}_\phi\circ R)^\prime\right\|_r \|(R^\prime)^{-1}\|_r\le
\left\|f^{(i)}_\phi\circ R\right\|_{r+1} \|(R^\prime)^{-1}\|_r.
$}
\] 
Iterating the inequality shows that $\|f^{(i+1)}_\phi\circ R\|_r$ is bounded above by
\[
\mbox{\small$
\left\|f^{\prime}_\phi\circ R\right\|_{r+i} \|(R^\prime)^{-1}\|^i_r=
\left\|\frac{I^\prime}{R^\prime}\right\|_{r+i} \|(R^\prime)^{-1}\|^i_r\le
\left\|I^\prime\right\|_{r+i} \|(R^\prime)^{-1}\|_{r+i} \|(R^\prime)^{-1}\|^i_r.
$}
\]
Now
\[
\mbox{\small$
\displaystyle{
|I^\prime|=|\sin\phi+\cos\phi\, f^\prime|\le 1+\max_J|f^\prime|\quad\text{and}\quad 
|I^{(l)}|=|\cos\phi\, f^{(l)}|\le\max_J|f^{(l)}|\ \text{for\ } l\ge 2. 
}
$}
\]
so that $\|I\|_{r+i}\le D + 1+ \|f\|_{r+i}$. To bound $\|(R^\prime)^{-1}\|_k$ write
\[
\mbox{\small$
\displaystyle{
\left((R^\prime)^{-1}\right)^{(l)}\cos\phi=
\left(
\frac{1}{1-\tan\phi\, f^\prime}
\right)^{(l)}=
\left(\sum_{j\ge 0} (\tan\phi\, f^\prime)^j\right)^{\!\!(l)}=
\sum_{j\ge 0} \left((\tan\phi\, f^\prime)^j\right)^{(l)}
}
$}
\]
hence
\[
\mbox{\small$
\displaystyle{
\cos\phi\,\|(R^\prime)^{-1}\|_k\le
\sum_{j\ge 0}\left\|(\tan\phi\, f^\prime)^j\right\|_{k}\le
\sum_{j\ge 0}\left\|\tan\phi\, f^\prime\right\|_{k}^j=\frac{1}{1-\|\tan\phi\, f^\prime\|_k}.
}
$}
\]
Thus $\|f_\phi\|_r$ is bounded above by
\[ 
\mbox{\small$
\displaystyle{
\|f_\phi\|_0+
\sum_{i=0}^{r-1}
\|f^{(i+1)}_\phi\circ R\|_r\le
D+\|f\|_{0}+
\sum_{i=0}^{r-1}
\frac{D + 1+ \|f\|_{r+i}}{(\cos\phi-\|\sin\phi\, f\|_{r+i})
(\cos\phi-\|\sin\phi\, f\|_{r})^i}.
}
$}
\]
and since $\|f\|_{0}$,  $f\|_{r+i}$, $\|f\|_{r}$ are bounded above by $\|f\|_{2r-1}$
the claim follows. 
\end{proof}

\section{Smoothing a hinge}
\label{sec: smoothing a hinge}

A point $p$ on a $C^\infty$ plane curve is {\em infinitely flat\,}
if in a local chart that takes $p$ to $0$
and the tangent line at $p$ to the $x$-axis the curve
becomes the graph of a function $g$ with $g^{(r)}(0)=0$ for 
all $r\in N$.

A {\em hinge\,} is the union of two straight line segments
in $\R^2$ that have a common endpoint, which we call an 
{\em apex\,}. A {\em $(l, r, \a)$-hinge\,} is
a hinge with segments of lengths $l$, $r$
and angle $\a$ at the apex. 

A {\em convex smoothing of the hinge\,} 
is a convex $C^\infty$ plane curve that joins the endpoints of the hinge
and is infinitely flat at the endpoints. 

A plane curve joining two distinct points 
is {\em modelled on a function $f$ near endpoints\,} if each endpoint of
the curve has a neighborhood that can be taken by a rigid motion of $\R^2$ 
to a portion of the graph of $f$.

The main result of this section is Lemma~\ref{lem: smooth hinge} 
which constructs a convex smoothing of a hinge modelled on a certain function
near the endpoints and such that all of its derivatives are controlled. 
The existence of a convex smoothing prescribed near the endpoints
is nearly obvious. It is less clear how to
control the $C^r$ norms of the smoothing and the methods we develop 
for the purpose are lengthy and rather delicate. If we start with a hinge of fixed
sidelength and make the angle at the apex close to $\pi$
it is believable that we can find a smoothing with small $C^r$ norm.
The main difficulty is that
we have to work with hinges whose diameters tends to zero, and hence to retain convexity
we have to shorten the prescribed portion of the graph near the endpoints. Doing so
may a priori increase the $C^r$ norm of the smoothing, and thus we need a uniform estimate on the $C^r$ norm that
persists when the hinge gets smaller. We overcome this difficulty via an explicit
``partition of unity'' construction. An apparent deficiency of this method is that it fails to 
prescribe the smoothing arbitrarily near the endpoints, where we have to assume that
the curve can be rotated to the graph of a pliable function, as defined below.

A sequence $(a_k)$ {\em grows at most exponentially\,} if 
the sequence $2^{k\g} a_k$ is bounded for some $\g\in\R$.
A sequence $(a_k)_{_{k\ge K}}$ {\em super-exponentially converges\,} to $a$ 
if for every $\g\in \R$ we have $2^{k\g}(a_k-a)\xrightarrow[k\to\infty]{} 0$.

Given $k_0\in \mathbb Z$ and an endpoint $c$ of a compact interval $J\subset \R$, 
we call $f\in C^\infty(J)$
{\em pliable at $c$\,} 
if $\mbox{\small$f={\displaystyle\sum_{k\ge k_0} c_k g_k}$\!\!}$  where\vspace{-3pt}
\begin{itemize}
\item[(i)]
$(c_k)$ is a sequence of positive reals that super-exponentially
converges to $0$.\vspace{2pt}
\item[(ii)]
$g_k\in C^\infty(J)$ 
such that $\{\supp g_k\}_{_{k\ge k_0}}$
is a locally finite family of compact subsets of $J\setminus\{c\}$,
and for every $r$ the sequence of the
$C^r$ norms of $g_k$ grows at most exponentially.\vspace{2pt}
\item[(iii)]
there is  $L>0$ such that if 
$0<2\e<\text{length of}\ J$, and $J_\e$ is the set of
indices $k$ such that $\supp g_k$ intersects $\{x\in J: \e\le |x-c|\le 2\e\}$,
then $J_\e$ has at most $L$ elements and $2^{-kL}\le\e$
for every $k\in J_\e$. 
\end{itemize}

\begin{lem}
\label{lem: pliable}
If $\mbox{\small
${f=\displaystyle\sum_{k\ge k_0} c_k g_k}$
\!\!}$ 
is pliable at $c$, then the following hold:\vspace{-4pt}
\begin{itemize}
\item[(a)] 
The partial sums $\mbox{\small${\displaystyle\sum_{k=k_0}^l c_k g_k}$\!\!}$ 
converge to $f$ in the $C^r$ norm for each $r\in\N$. \vspace{2pt}
\item[(b)] 
$f^{(r)}(c)=0$ and $\mbox{\small${\displaystyle f^{(r)}=\sum_{k\ge k_0} c_k g_k^{(r)}}$\!\!}$ is pliable at $c$ for each $r\in \N$. \vspace{3pt}
\item[(c)] 
$hf+g$ is pliable at $c$ for every $h, g\in C^\infty(J)$  with $\supp\, g\subset J\setminus\{c\}$.\vspace{2pt}
\item[(d)]
$f\circ\iota$ is pliable at $\iota^{-1}(c)$ for each isometry $\iota$ of $\R$.\vspace{2pt}
\item[(e)] 
If $d\co K\to J$ is a diffeomorphism, then $f\circ d$ is pliable at $d^{-1}(c)$.
\end{itemize}
\end{lem}
\begin{proof} (a), (c), (d) are straightforward, and (a) implies (b). Let us prove (e).
We use the series $\mbox{\small${\displaystyle f\circ d=\sum_{k\ge k_0} c_k\,  g_k\circ d}$\!\!}$ for which (i) is trivial. 
To check (ii) apply the chain rule iteratively to see that
the $C^r$ norms of $g_k\circ d$ grow at most exponentially.
 
It remains to verify (iii).
By composing with isometries of $\R$ we can assume that $c=0=d^{-1}(c)$ and $K$, $J$ consists of nonnegative reals.
We can also assume $k_0\ge 1$ else by (c) we simply add to $f$ the partial sum 
$\mbox{\small$-{\displaystyle\sum_{k=k_0}^0 c_k  g_k}$\!\!}$,
and then adjust $f\circ d$ accordingly.
Fix $\e>0$ with $2\e$ less that the length of $K$. 
Note that $\mathrm{supp}\, g_k\circ d$ intersects $[\e,2\e]$
if and only if $\mathrm{supp}\, g_k$ intersects $[d(\e), d(2\e)]$.
Fix any $\l$ such that $d$ and $d^{-1}$ are $\l$-Lipschitz.
Let $l$ be the least positive integer with $2^{-l}d(2\e)\le d(\e)$. 
Since $d(2\e)\le 2\l\e$ and $d(\e)\ge \e/\l$ we get $\frac{d(2\e)}{d(\e)}\le 2\l^2$
so $l\le l_\l=1+\log_2(2\l^2)$.
Since $f$ is pliable at $0$ and
\[
\left[d(\e), d(2\e)\right]\subset\bigcup_{s=1}^l [2^{-s}d(2\e), 2^{1-s}d(2\e)]
\]  
for at most $lL$ values of $k$ the support of $g_k$ intersects $[d(\e), d(2\e)]$
and each of these $k$ satisfies $2^{-kL}\le\frac{d(2\e)}{2}\le \l\e$.
Thus for $L_\l=\max\left(l_\l L, L+|\log_2 \l|\right)$ we have $L_\l\ge lL$ and
$2^{-kL_\l}\le \e$, which proves (iii). 
\end{proof}

\begin{lem}
\label{lem: f_phi'' pliable}
Suppose $f$ is pliable at $c$ and $f_\phi$ is defined. Then 
$f_\phi^{\prime\prime}$ is pliable at $R(c)$ and
if $\phi\neq 0$, then $f^\prime_\phi$ is not pliable at $R(c)$.
\end{lem}
\begin{proof} Consider the formulas (\ref{form: 1st 2nd deriv rotation}).
Since $f^\prime(c)=0$, we get
$f^\prime_\phi(R(c))=\tan\phi\neq 0$ so $f^\prime_\phi$ is not pliable at $R(c)$
by Lemma~\ref{lem: pliable}(b). Since $f^{\prime\prime}$ is pliable at $c$, so is
$f_\phi^{\prime\prime}$ thanks to Lemma~\ref{lem: pliable}(c), (e).
\end{proof}

\begin{rmk}
If $f$ is pliable at $0$, then $R(0)=0$, hence $f_\phi^{\prime\prime}$ is also
pliable at $0$. 
\end{rmk}

\begin{lem} 
\label{lem: smooth hinge}
Let $f\in C^\infty([0,\tau])$ be pliable at $0$ and
such that $f^\prime$ and $f^{\prime\prime}$ are positive on $(0,\tau]$. 
Then there exists a sequence of $(l_m, r_m, \a_m)$-hinges, $m>0$, 
and a sequence of their convex smoothings such that
\begin{itemize}
\item each smoothing is 
modelled on $f$ near the endpoints,
\item 
each smoothing has positive curvature everywhere except at the endpoints, 
\item
each smoothing is the graph of a convex function $F_m$ and 
for each $r\in\Z_{\ge 0}$ the $C^r$ norms of $F_m$
are uniformly bounded,\vspace{2pt}
\item 
$\displaystyle{\sum_{m>0} 2^m(l_m+r_m)<\infty}$ 
and $\displaystyle{\sum_{m>0}\,}2^m(\pi-\a_m)=\frac{\pi}{n}$ for some integer $n\ge 2$.
\end{itemize}
\end{lem}
\begin{proof} 
Since $f$ is pliable at $0$, $f^{(r)}(0)=0$ for all $r\in\Z_{\ge 0}$,
and $0$ is the only point where the graph of $f$ is infinitely flat because
$f^{\prime\prime}>0$ at other points.

For positive reals $d$, $\g$ with $4d<\tau$ and 
$\g<\frac{\pi}{3}$ 
let $h=d\tan\g$ and
consider the hinge $V$ in $\C$ with apex $0$ and endpoints 
$u=-d+ih$, $v=d+ih$; thus $V$ has sidelengths 
$\frac{d}{\cos\g}$.

Translating the graph of $f$ to the left by $\frac{d}{\cos\g}$
and then rotating it clockwise by $\g$ results in 
a curve that is tangent to $V$ at $u$
and is the graph of a convex function which we denote $f_u$.

Reflecting the graph of $f$ with respect to the $y$-axis, translating
the result to the right by $\frac{d}{\cos\g}$, and then rotating it 
counterclockwise by $\g$ gives a curve that is tangent to $V$ at $v$
and is the graph of a convex function which we denote $f_v$.

Our assumptions of $d$, $\tau$ and $\g$ imply that
the domains of $f_u$, $f_v$ contain $[-d,d]$.
Also $-f_u^\prime(-d)=\tan\g=f_v^\prime(d)$.
Since a graph of function has zero curvatures precisely where
its second derivative vanishes, 
$f_u^{\prime\prime}$, $f_v^{\prime\prime}$ are positive away from $-d$, $d$,
respectively. By Lemma~\ref{lem: f_phi'' pliable} 
$f_u^{\prime\prime}$, $f_v^{\prime\prime}$ are pliable at $-d$, $d$, respectively, so that
$f_u^{(r)}(-d)=0=f_v^{(r)}(d)$ for each $r\ge 2$. 

Let $\Phi\in C^\infty(\R)$ be a bump function with
$\Phi\vert_{[-\frac{1}{2},\frac{1}{2}]} = 1$, 
and $\supp\Phi=[-1,1]$. 
For $0<4\e<d$ and $-d<x<d$ consider the functions
\[
\Phi_u^\e(x)=\Phi\left(\frac{d+x}{2\e}\right) \qquad
\Phi_v^\e(x)=\Phi\left(\frac{d-x}{2\e}\right) \qquad 
\Phi_0^\e(x)=\Phi\left(\frac{\e}{d-|x|}\right).
\]
Note that $\Phi_u^\e$, $\Phi_v^\e$, $\Phi_0^\e$
are constant near $\pm d$ and $0$, so
they extend to $C^\infty$ functions 
on $\R$ that are constant for $|x|\ge d$.  
This allows us to think of
$f_u^{\prime\prime}\Phi_u^\e$, 
$f_v^{\prime\prime}\Phi_v^\e$ as $C^{\infty}$ functions on $\R$.

Let $F_\e$ be the solution of 
\begin{equation}
\label{eq: F''}
F^{\prime\prime} = f_u^{\prime\prime}\Phi_u^\e + 
f_v^{\prime\prime}\Phi_v^\e + b_\e\Phi_0^\e
\end{equation}
subject to the initial conditions $F^{\prime}(-d) = f_u^{\prime}(-d)$,
$F(-d) = f_u(-d)$. Here $b_\e$ is a constant
for which $F^\prime(d)=f_v^\prime(d)=\tan\g$, i.e., $b_\e$ is 
the unique solution of
\begin{equation}
\label{eq: F'}
f_v^\prime(d)=F_\e^\prime(d)=F_\e^\prime(-d)+
\int\limits_{\!\!-d}^{\,d} 
\!\!\left( f_u^{\prime\prime}\Phi_u^\e + f_v^{\prime\prime}\Phi_v^\e\right)
\,+\,b_\e\!\int\limits_{\!\!-d}^{\,d}\! \Phi_0^\e\
\end{equation}
which exists because the integral of $\Phi_0^\e$ over $[-d,d]$ is positive.
For the rest of the proof we choose $\e$ as in the following lemma.

\begin{lem}
\label{lem: epsilon} 
For every sufficiently small $\g$ there exists a unique $\e$
such that $f^\prime(4\e)=\tan\g$ and $f^\prime_u(2\e-d)<0<f_v^\prime(d-2\e)$.
\end{lem}
\begin{proof}
Suppose $\g$ is so small that
$\tan\g$ is in the range of $f^\prime$.
Then there clearly exists a unique
$\e$ with $f^\prime(4\e)=\g$, and we are
going to show that the other inequalities also hold.
Since $f_u(t)=f_v(-t)$
for all $t$ in the domain of $f_u$, 
we have $f^\prime_u(2\e-d)=-f_v^\prime(d-2\e)$,
so it suffices to show that $f^\prime_u(2\e-d)<0$.
To ease notations set $h(x)=f(x+\frac{d}{\cos\g})$.
The graph of $f_u$ is the image of the curve 
$x\to e^{-i\g}(x+ih(x))$ in $\C$
whose real and imaginary parts are
\[
R(x)=x\cos\g+h(x)\sin\g\quad\text{and}\quad 
I(x)=-x\sin\g+h(x)\cos\g
\] 
respectively. Thus $f_u\circ R=I$;  
differentiating this identity gives
\[
f_u^\prime(R(x))=\frac{I^\prime(x)}{R^\prime(x)}=
\frac{-\sin\g+h^\prime(x)\cos\g}
{\cos\g+h^\prime(x)\sin\g}
\]
where clearly $R^\prime>0$. 
Hence $f_u^\prime(R(x))\le 0$ if and only if 
$h^{\prime}(x)\le \tan\g$.
Since $f^\prime$ and $h^\prime$ have the same range,
there is a unique $y_\e$ with $h^\prime(y_\e)=\tan\g$.
The functions $h^\prime$ and $R$ are non-decreasing, hence
$h^{\prime}(x)\le \tan\g$
is equivalent to $R(x)\le R(y_\e)$. We conclude that 
$f_u^\prime(2\e-d)\le 0$ if and only if $2\e-d\le R(y_\e)$.
Since $y_\e\cos\g< R(y_\e)$ the desired inequality $f_u^\prime(2\e-d)< 0$
would follow from  $2\e\le d+y_\e\cos\g$ which after dividing by $\cos\g$
and applying $f^\prime$ is equivalent to
\begin{equation*}
f^\prime\!\left(\frac{2\e}{\cos\g}\right)\le 
f^\prime\!\left(\frac{d}{\cos\g}+y_\e\right)=
h^\prime(y_\e)=\tan\g=f^\prime(4\e)
\end{equation*}
which holds because $\frac{1}{\cos\g}<2$. This proves
Lemma~\ref{lem: epsilon}.
\end{proof}

For such $\e$ the constant $b_\e$ is positive because 
$f^\prime_u(2\e-d)<0<f_v^\prime(d-2\e)$ implies
\begin{equation}
\label{eq: bounds on double prime term}
\int\limits_{\!\!-d}^{\,d} 
 f_u^{\prime\prime}\Phi_u^\e\, <
\int\limits_{\!\!-d}^{2\e-d} 
\!\! f_u^{\prime\prime} <
-f_u^\prime(-d)
\qquad\text{and}\qquad
\int\limits_{\!\!-d}^{\,d} 
f_v^{\prime\prime}\Phi_v^\e\, < \int\limits_{d-2\e}^{\,d} 
f_v^{\prime\prime} <
f_v^\prime(d)
\end{equation}
which together with (\ref{eq: F'}) shows that 
the integral of $b_\e\Phi_0^\e$ over $[-d,d]$ is positive.

Thus each summand in (\ref{eq: F''}) is nonnegative. 
Furthermore $\Phi_u^\e$, $\Phi_v^\e$, $\Phi_0^\e$ has no common zeros,
hence $F_\e^{\prime\prime}\vert_{(-d, d)}>0$ and $F_\e^{\prime\prime}(-d)=
0=F_\e^{\prime\prime}(d)$.

The $C^{\infty}$ function $F_\e^{\prime\prime}$ restricts
to $f_u^{\prime\prime}$, $f_v^{\prime\prime}$ on the $\e$-neighborhoods
of $-d$, $d$, respectively. In view of the initial conditions
$F_\e=f_u$ on the $\e$-neighborhood of $-d$, and 
since $F_\e^\prime(d)=f_v^\prime(d)$, the function
$F_\e-f_v$ is constant on the $\e$-neighborhood of $d$.

The graph of $F_\e$ generally need not be tangent to
the hinge $V$ at $d$. On the other hand
the tangent lines to the graph of $F_\e$ at $-d$, $d$ intersect
so that the straight line segments joining the intersection
point with $F_\e(-d)$, $F_\e(d)$ form a hinge, denoted $V_{F_\e}$, and
the graph of $F_\e$ is a convex smoothing of $V_{F_\e}$ that is modelled on $f$.

By the triangle inequality the sum of sidelengths
of $V_{F_\e}$ is $\le \frac{4d}{\cos\g}$. 
(In fact, a plane geometry argument shows that the sum of sidelengths
of $V_{F_\e}$ and of $V$ are the same, i.e., $\frac{2d}{\cos\g}$ 
but we do not need this here).

Fix $r$ let us find bounds on the $C^r$ norm of $F_{\e}$. 
A sketch of $V$, $V_{F_\e}$ and the graph of $F_\e$ reveals that
$|F(t)|\le 3d\tan\g$ for $t\in [-d,d]$. To estimate $F^\prime$ write
\begin{equation}
\label{eq: F'(t)}
F_\e^\prime(t) =F_\e^\prime(-d)+ \int_{-d}^{t} 
\!\!\left( f_u^{\prime\prime}\Phi_u^\e + f_v^{\prime\prime}\Phi_v^\e\right)\, +\, 
b_\e\!\int_{-d}^{t}\!\! \Phi_0^\e
\end{equation}
and note that on the right hand side the first summand equals $-\tan\g$,
the second summand is within $(0,2\tan\g)$ because of (\ref{eq: bounds on double prime term}),
and the third summand is at most $2d\,b_\e$. To estimate $b_\e$ note that
the integral of $\Phi_0^\e$ over $[-d,d]$ is within
$(2d-4\e, 2d)$ so that (\ref{eq: F'(t)}) implies 
$b_\e\le\frac{\tan\g}{d-2\e}<2\frac{\tan\g}{d}$.
Thus
\[
\max_t|F_\e^\prime(t)|<\tan\g+2\tan\g+4d\frac{\tan\g}{d}=
7\tan\g
\]
and it remains to bound the $C^{r-2}$ norm of $F_\e^{\prime\prime}$.

Iterated derivatives of $b_\e\Phi_0^\e$ yield
terms $b_\e\,\,\e^m\,(d-|x|)^{-k}\,c\Phi_0^{(l)}$ where $|c|$, $k$, $l$, $m$
are positive integers depending only on the order of differentiation. 
Since $\Phi_0$ vanishes on
the $\e$-neighborhoods of $\pm d$, the above term is bounded above
by $b_\e\,\e^{m-k}\,|c\,\Phi_0^{(l)}|$. This is uniformly bounded in $\e$
because $\tan\g=f^\prime(4\e)$ and  all derivatives of $f$ vanish
at $0$.

The first two summands of (\ref{eq: F''}) are treated in the same way
so we focus on $f_u\Phi_u^\e$. 
A bound on $|f_u^{(r)}|$ in terms of $|f^{(r)}|$ and 
some upper bound on $d$ and $\g$ follows from Lemma~\ref{lem: rotation}; the same bound works for
$|f_u^{(r)}\Phi_u^\e|\le |f_u^{(r)}|$. 

It remains to consider
the terms $(\Phi_u^\e)^{(s)}f_u^{(l)}$ with $s>0$ and $l\ge 2$.
Note that $(\Phi_u^\e)^{(s)}=(2\e)^{-s}\Phi^{(s)}$
vanishes outside $(\e-d, 2\e-d)$, and
$f_u^{(l)}$ is pliable at $-d$ by Lemma~\ref{lem: f_phi'' pliable}. 
Using notations from the 
definition of a pliable series we write
\[
(\Phi_u^\e)^{(s)}f_u^{(l)}=\sum_{k\in J_\e} c_k g_k \frac{\Phi^{(s)}}{(2\e)^{s}}
\] 
where $|J_\e|\le L$ and $\e\ge 2^{-kL}$ for each $k\in J_\e$.
By local finiteness of $\{\supp g_k\}$
the smallest $k\in J_\e$ tends to $\infty$
as $\e\to 0$.
For every $k\in J_\e$ the summand $|c_k (2\e)^{-s} g_k \Phi^{(s)}|$ is bounded above
by $c_k 2^{kLs-s}|g_k\Phi^{(s)}|$ which tends to $0$ as $\e\to 0$.
Since we have at most $L$ such summands, we get
an upper bound on $|(\Phi_u^\e)^{(s)}f_u^{(l)}|$ that is independent of $\e$.

To complete the proof fix any $d_*\in (0,\frac{\tau}{4})$, $\g_*\in (0,\frac{\pi}{3})$,
and any converging series $\sum_{m>0} 2^m d_m$ with $d_m\in (0,d_*)$.
We showed above that for any positive integers $r$, $m$
there  is $\g_{m,r}<\g_*$ such that for any $\g_m\in (0,\g_{m,r})$
and $\e_m$ with $f^\prime (4\e_m)=\tan\g_m$
the $C^r$ norm of the function $F_{\e_m}$ constructed 
for $(d,\e, \g)=(d_m, \e_m, \g_m)$
is bounded uniformly in $m$.
Passing to the 
diagonal subsequence corresponding to the angle $\g_{m,m}$ 
gives  $F_{\e_m}$ for which every $C^r$ norm is uniformly bounded.
The sum $\sum_{m>0}2^{m+1}\g_m$ takes every value in
$\left(0, \sum_{m>0}2^{m+1}\g_{m,m}\right)$ and in particular, the value $\frac{\pi}{n}$
for some sufficiently large positive integer $n$. Set $\a_m=\pi-2\g_{m}$.
The sum of the sidelengths $l_m+r_m$ of the hinge smoothed by $F_{\e_m}$
is at most $4d_m$, so $\sum_{m>0}2^m(l_m+r_m)$ converges. 
\end{proof}

\begin{rmk}
\label{rmk: finite collection}
The above proof shows that given any finite collection of functions $f$
satisfying the assumptions of Lemma~\ref{lem: smooth hinge}
there is a sequence $(l_m, r_m, \a_m)$ such that the conclusion of 
the lemma holds for every $f$ in the collection. 
\end{rmk}

\section{Building the curve}
\label{sec: building a curve}

In this section we piece together the hinges produced in Section~\ref{sec: smoothing a hinge}
to form a $C^\infty$ closed convex curve. In doing so we are able to prescribe the Gauss image
of the set of points of zero curvature, which will be a Cantor-like set. The schematic below 
illustrates the construction.

\begin{figure}[h]
\caption{The thickened curves
represent the smoothings that fit together following the pattern
of the middle third intervals in the complement of the standard Cantor set.}
\centering
\includegraphics[scale=.6]{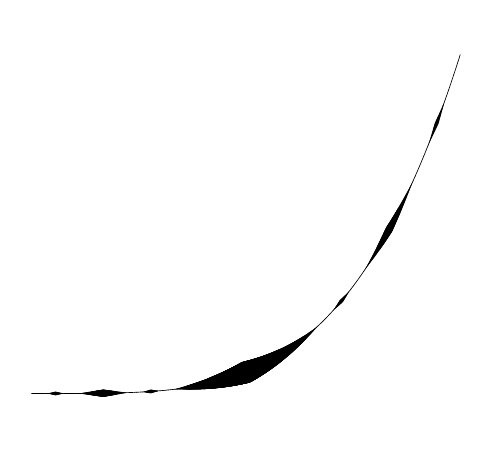}
\end{figure}

Fix a sequence of hinges modelled near the endpoints on a function $f$
as in Lemma~\ref{lem: smooth hinge},
and build a sequence of smooth convex curves $J_m$ in $\R^2$ as follows.

On the $x$-axis in $\R^2$ fix a 
closed interval $J_0$ of length $\sum_{m>0} 2^m(l_m+r_m)$.
Let $p$ be the left endpoint of $J_0$.
Inside $J_0$ mark the middle open interval $I_0^1$ of length $l_1+r_1$,
where ``middle'' means that $J_0$ and $I_0^1$ have the same midpoint.
Also mark a point $p_0^1\in I_0^1$ that divides
$I_0^1$ into two intervals, the left one of length $l_1$
and the right one of length $r_1$.
Bend $J_0$ at $p_0^1$ into a hinge with angle $\pi-2\g_1$
while keeping the basepoint $p$ fixed.
Then $I_0^1$ becomes a hinge, and we
replace the hinge with the convex smoothing given by Lemma~\ref{lem: smooth hinge}.
Call the result $J_1$.

Outside the smoothing $J_1$ consists of two
straight line segments. On each of them mark the middle open interval
of length $l_2+r_2$, denote the intervals $I_1^1$, $I_1^2$, 
and mark point $p_1^1$, $p_1^2$ such that each $p_1^s$ divides
$I_1^s$ into two intervals with the left one of length $l_2$
and the right one of length $r_2$. Bend each $I_1^s$ into a hinge
that forms the angle $\g_2$ with $J_1$ so that
$p$ stays fixed.
Replace each hinge with
the convex smoothing given by Lemma~\ref{lem: smooth hinge}.
Call the result $J_2$.

Iterating the construction on the $m$th step we bend $2^m$ segments
into hinges while keeping $p$ fixed
and replace each of the by convex smoothings, 
which yield a curve $J_m$. 

Each $J_m$ is the graph of a convex $C^\infty$ function $h_m$ whose
$C^r$ derivatives are bounded uniformly in $m$ by
Lemmas~\ref{lem: rotation}, \ref{lem: smooth hinge}. 
Because of the way we bend we have $h_m\le h_{m+1}$ for each $m$, hence
the sequence $(h_m)$ has a limit $h$
which is therefore $C^\infty$ and convex. In fact, it is strictly
convex as by construction it contains no straight line segments.

Once a convex smoothing appears in some $h_m$ its slighly rotated copies show up in
every $h_k$ with $k>m$, and hence eventually in $h$.

The graph of $h$ is infinitely flat at the endpoints. (One way to see it is to start
we the $1$-neighborhood of $J_0$ in place of $J_0$, run the above
procedure over $J_0$ carrying along the two flat segments on the sides.
The result is a $C^\infty$ extension of the graph of $h$ with 
a straight line segments on each side).

The Gauss map takes the graph of $h$ to a circular arc of length $\frac{\pi}{n}$.
Moreover, the images of smoothings are taken to removed intervals of a Cantor-like set, 
and the reflection of $\R^2$ in the axis of symmetry
of this circular arc is an involution on the set of removed intervals.

Piecing together
$2n$ copies of the graph of $h$ gives a $C^\infty$ closed strictly convex curve,
which we denote $C_f$. The curve has rotational symmetry of order $2n$.
The curvature vanishes precisely on the closure
of the set of the endpoints (because the curvature does not vanish
on the interior of the smoothings, and what remains is  Cantor-like set
in which the set of endpoints is dense).
It is convenient to rotate $C_f$ so that the 
Gauss map takes the base point $p$ to $1\in S^1$.

Let $Z_f$ denote the Gauss image of the set of points of zero curvature of $C_f$.
Up to rotation $Z_f$ depends only on the sequence $(\a_m)$.
Also $Z_f$ has an order $2n$ rotational symmetry, and $Z_f^{-1}=Z_f$
due to the above-mentioned axial symmetry of the Gauss map image
of the graph of $h$.

Let us show that for every real $\a$ the sets $e^{i\a}Z_f$ and $Z_f$ intersect,
i.e. $Z_f\cdot Z_f^{-1}=S^1$.  Let $O$ be the preimage
of $Z_f$ under the projection $[0,2\pi]\to S^1$ given by $t\to e^{it}$.
It is enough to show that $O+O\supset [0,2\pi]$. By construction
$O=\bigcup_{l=1}^{2n} O_l$ where $O_{l}\supset\frac{\pi (l-1)}{n}+O_1$ 
and $O_1$ is constructed from $\left[0,\frac{\pi}{n}\right]$
be removing middle thirds of lengths $\g_0, \g_1, \g_2, \dots$.
We can choose $\g_m$ as small as we like, and in particular,
we can arrange that at each step the ratio of the lengths
of each new interval (that remains after removing the middle third)
and the old interval is any number in 
$\left(0,\frac{1}{2}\right)$. By~\cite[Corollary 3.3]{CHM} 
if the ratio is at least $\frac{1}{3}$ on every step, then 
$O_1+O_1=\left[0,\frac{2\pi}{n}\right]$. It follows that
$O_l+O_1\supset\frac{\pi l}{n}+\left[0,\frac{2\pi}{n}\right]$
so that $O+O\supset [0,2\pi]$.

For a convex smoothing produced in Lemma~\ref{lem: smooth hinge} we refer to
its endpoint over $-d$ as the {\em left endpoint}.
Let $E_f\subset Z_f$ be the image under the Gauss map of the set of left endpoints
of convex smoothings that form $C_f$; thus $E_f$ is a countable dense subset of $Z_f$.

\section{Proof of Theorem~\ref{intro: main}}
\label{sec: proof of main thm}

For every $\a\in (0,1)$ 
Boman constructed in~\cite{Bom-sum} two
strictly convex functions $f, g\in C^\infty(\R)$ whose graphs are infinitely flat at
$0$, and such that $h=f\Box g$ is not $C^{4,\a}$.
Specifically, starting from $f_k(x)=a_k^2\frac{x^2}{2}$ patched as in Lemma~\ref{lem: boman} below
to produce $f$, and similarly using $g_k(x)=\frac{x^4}{4}$ to produce $g$, Boman showed that
if $b_k$, $t_k$ are as in Lemma~\ref{lem: boman}  and 
$\frac{a_k^\a}{b_k}\xrightarrow[k\to\infty]{} 0$, then near $2t_k$ the $C^{4,\a}$ norm of $h=f\Box g$
is unbounded as $k\to\infty$.

Consider the curves $C_f$, $C_g$ built as in Section~\ref{sec: building a curve} with inputs $f$, $g$,
respectively. Let $A$, $B$ be convex sets enclosed by $C_f$, $C_g$.
By Remark~\ref{rmk: finite collection}
we may assume that $Z_f=Z_g$ so that $Z_f$ cannot be rotated off $Z_g$.
Also their sets of the left endpoints coincide: $E_f=E_g$. 
If $R\in E_f E_g^{-1}$, then $R$ superimposes a copy of the graph of $f$ over a copy of 
the graph of $g$, and hence $R(A)+B$ is not $C^{4,\a}$ thanks to the above-mentioned result
in~\cite{Bom-sum}. 

It suffices to show that for every $l\in\N$
any point of $E_f E_g^{-1}$ has an open neighborhood in $S^1$
consisting of rotations $Q$ such that the boundary of $Q(A)+B$ either is not $C^{4,\a}$ 
or its $C^{4,\a}$ norm is $>l$.
(Indeed, if $U_l$ is the union of such neighborhoods, then 
$\displaystyle{\bigcap_{l} U_l}$ is a $G_\de$ subset that contains $E_f E_g^{-1}$. 
For every $Q\in \displaystyle{\bigcap_{l} U_l}$ the set $Q(A)+B$ is not $C^{4,\a}$.
Since any set containing a dense $G_\de$ subset is comeager, 
the conclusion of Theorem~\ref{intro: main} follows.)
 
Fix an arbitrary $R_1\in E_f E_g^{-1}$, and set $A_1=R_1(A)$. 
The boundary $A_1+B$ has a tangent line over which
(a portion of) the boundary is (a portion of) the graph of $f\Box g$. 
(Here we identify the tangent line with the $x$-axis on $\R^2$
so that the point of tangency is the origin.)
Fix $k$ such that near $2t_k$ the $C^{4,\a}$ norm of $f\Box g$ is $>l$. 
(It is worth mentioning that when the diameter of the smoothing is small,
so is $t_k$, hence there is no universal $t_k$ that works for every $R_1$.)
Henceforth we vary $x$ near $2t_k$.
Consider $e^{i\de} A_1+B$ where $|\de|\ll \min(b_k, t_k)$; thus we are to rotate the graph of $f$ while keeping $g$ fixed.
Write 
\[
(s+if(s))e^{i\de}=R(s)+iI(s)=y+if_\de(y)\quad\text{so\ that}\quad f_\de=I\circ R^{-1}
\]
where $R(s)=y$, $I(s)=f_\de(y)$ are as in Section~\ref{sec: rotated graphs}.
Let $h_\de=f_\de\Box g$, i.e. 
$h_\de(x)=f_\de(y_x)+g(x-y_x)$ where $y_x$ is a unique solution of
\begin{equation}
\label{form: inf}
f_\de^\prime(y)=g^\prime(x-y).
\end{equation}
Since $f^{\prime\prime}$, $g^{\prime\prime}$
are nonnegative and only vanish at $0$, the formula (\ref{form: 1st 2nd deriv rotation})
for $f_\de^\prime$, $f_\de^{\prime\prime}$ implies that 
the derivative of $f_\de^{\prime}(y)-g^\prime(x-y)$ by $y$
only vanishes if $x=0=y$, which only solves (\ref{form: inf}) for $\de=0$.
Hence by the Implicit Function Theorem the solution $y=y(x,\de)$
of (\ref{form: inf}) is $C^\infty$ near any $(x, \de)\neq (0,0)$, and in particular, near
$(2t_k, 0)$. Because of (\ref{form: h via y(x)}) we have $h_\delta^\prime(x)=f^\prime(y(x,\de))$, and  
hence $h_\de$ is $C^\infty$ near $(2t_k, 0)$.
Since the $C^{4,\a}$ norm of $h_0=f\Box g$ near $2t_k$ is $>l$,
the same is true for $h_\de$ for all small $\de$. 
This completes the proof of Theorem~\ref{intro: main}.

\begin{rmk} Let us justify the claim made
after the statement of Theorem~\ref{intro: main} that the proof can be modified 
to arrange that $A=B$. The only change will be in Section~\ref{sec: building a curve}
where we assemble the curve starting from a Cantor-like set on an interval. 
There let us insert a smoothing modelled on $f$ at odd-numbered steps, and 
a smoothing modelled on $g$ at even-numbered steps. Let $E_f$ be the Gauss map image
of the set of left endpoints of the smoothings modelled on $f$; define $E_g$ similarly.
The only thing that needs verifying is that that both $E_f$ and $E_g$ are dense 
in the Cantor-like set on $S^1$. As in  Section~\ref{sec: building a curve} 
we pull everything to $[0,2\pi]$. In a  standard way we think of a point $x$
in the Cantor-like set as a binary sequence, i.e., if $x$ lies in an interval 
that remained on step $k-1$, then the $k$th term in the sequence records
whether the middle portion of the interval removed on the $k$th step
is to the right or to the left of $x$. Thus the binary sequence specifies 
a sequence of removed intervals that converges to $x$. In the corresponding 
sequence of smoothings those modelled on $f$ and $g$ alternate. So $x$ is a limit
point of $E_f$ as well as of $E_g$.
\end{rmk}

\appendix

\section{Boman's lemma on convex patching}

The following lemma appears in~\cite[page 221]{Bom-sum}
except for the part (a) which is implicit in Boman's proof.
Since (a) is required for our purposes, 
we reproduce the proof (with some details added). 

\begin{lem} 
\label{lem: boman}
For each integer $k\ge 0$ 
let $b_k\in\R$ such that $\{2^k b_k\}$ 
is a decreasing sequence of positive numbers 
that  super-exponentially converges to $0$ as $k\to\infty$,
set $t_k = 4^{-k}$, 
let $f_k\in C^{\infty}([-t_k, t_k])$ such that $f_k^{\prime\prime}$ is positive, 
$f_k(0) = 0=f_k^\prime(0)$ and  
$\displaystyle{\sup_{x,k}}|f_k^{(r)}(x)| < \infty$ for each $r\ge 0$. 
Then there are $K>0$
and $f\in C^{\infty}(\R)$ such that\vspace{-6pt} 
\begin{itemize}
\item[\textup(a)] $f$ equals the sum of a series that is pliable at $0$.\vspace{4pt} 
\item[\textup(b)]
$f^\prime\vert_{(0,\,3t_K]}$ and $f^{\prime\prime}\vert_{(0,\,3t_K]}$ 
are positive, $f^{(r)}(0)=0$ for all $r$,\vspace{5pt}
\item[\textup(c)]
\label{eq: f' and b_k}
$f^\prime (x) = b_k f_k'(x - t_k) + b_k$ for all $k\ge K$ and all
$x$ with $|x - t_k| < t_{k+1}$.
\end{itemize}
\end{lem}
\begin{proof}
Fix $\Psi \in C^{\infty}(\mathbb{R})$ such that
$\Psi \geq 0$, $\Psi\vert_{[\frac{3}{4}, \frac{5}{4}]} = 1$, 
the set where $\Psi>0$ is precisely $(\frac{2}{3}, \frac{3}{2})$, and 
$\sum_{m\in\Z} \Psi (2^m x)=1$ for any $x>0$.
To arrange for the last property divide any $\Psi$
that satisfies the other three properties by
$\sum_{m\in\Z} \Psi (2^m x)$. This works because
\begin{itemize}
\item 
any positive number lies in either one or two intervals of the
form $2^{-m}(\frac{2}{3}, \frac{3}{2})$ so
that near each point the sum is finite and positive, 
and $[\frac{3}{4}, \frac{5}{4}]$
intersects exactly one of these intervals
so that $\sum_{m\in\Z} \Psi(2^m x)$ is a $C^\infty$ function 
which equals $1$ on
$[\frac{3}{4}, \frac{5}{4}]$ and is positive on $(0,\infty)$,
\item
substituting $x$ with $2^l x$, $l\in\Z$, does not change $\sum_{m\in\Z} \Psi (2^m x)$.
\end{itemize}

For an integer $k\ge 0$ set $\Psi_k(x) = \Psi(2^k x)$ so that
$\supp\Psi_{2k}= [\frac{2}{3}t_k, \frac{3}{2}t_k]$ and 
$\supp\Psi_{2k-1} =[\frac{4}{3}t_k, 3t_k]$. 
We search for $f$ by solving
\begin{equation} 
\label{eq:1}
f^{\prime\prime}(x) = 
{\displaystyle\sum_{k\geq K}} 
b_k f_k^{\prime\prime}(x-t_k) \Psi_{2k}(x) + 
\sum_{k \geq K}\alpha_k \Psi_{2k-1}(x)
\end{equation}
subject to the initial conditions $f(0)=0=f^\prime(0)$,
where $K$ and $\alpha_k$ are to be determined. 

Since $f_k^{\prime\prime}(x-t_k)$ is defined for $x\in [0,2t_k]$
and $\Psi_{2k}(x)$ vanishes outside $(\frac{2}{3}t_k, \frac{3}{2}t_k)$
we can think of $f_k^{\prime\prime}(x-t_k) \Psi_{2k}(x)$ as a $C^\infty$
function on $\R$.
The supports of $\Psi_{2k}$'s are disjoint, and so are the supports
of $\Psi_{2k-1}$'s, which implies that near each nonzero point
the right hand side of (\ref{eq:1}) has at most two summands, and 
therefore is $C^\infty$ away from $0$. 

On the interval $[\frac{3}{4}t_k, \frac{5}{4}t_k]$ we have $\Psi_{2k} = 1$
and $\Psi_{l}=0$ for any $l\neq 2k$, and therefore 
$f^{\prime\prime}(x) = b_k f_k^{\prime\prime}(x - t_k)$ on 
$[\frac{3}{4}t_k, \frac{5}{4}t_k]$. Therefore the part (c) of (\ref{eq: f' and b_k})
is equivalent to $f'(t_k) = b_k$. To achieve the latter identity we write
\[
b_{k-1} - b_k=
f'(t_{k-1}) - f'(t_k) = 
\int_{t_k}^{4t_k}f''(x) dx = b_{k}A_k + b_{k-1}B_k + \alpha_k  D_k\]
where the last equality comes from (\ref{eq:1}) and
{\footnotesize
\begin{equation*}
A_k=\int_{t_k}^{\frac{3}{2} t_k}\Psi_{2k}(x)f_k''(x-t_k) dx\qquad
B_k= \int_{\frac{8}{3}t_k}^{4t_k}\Psi_{2k-2}(x)f_{k-1}''(x-t_{k-1})dx\qquad 
D_k=\int_{\frac{4}{3}t_k}^{3t_k}\Psi_{2k-1}(x) dx
\end{equation*}}

because the only nonzero summands of $f^{\prime\prime}\vert_{[t_k, 4t_k]}$
are those containing 
$\Psi_{2k}$, $\Psi_{2k-1}$, $\Psi_{2k-2}$ 
whose respective supports $[\frac{2}{3}t_k, \frac{3}{2}t_k]$,
$[\frac{4}{3}t_k, 3t_k]$, $[\frac{8}{3}t_k, 6t_k]$
intersect $[t_k, 4t_k]$. Note that 
$A_k$, $B_k$, $D_k$ are nonnegative, being the integrals
of nonnegative functions, and moreover,
$D_k=2t_k\int_\R\Psi$ is positive
so we can write
\begin{equation}
\label{eq: alpha_k}
\a_k=\frac{b_{k-1} - b_k - b_k A_k - b_{k-1}B_k}{D_k}=
\frac{b_{k-1}(1-B_k)-b_k(1+A_k)}{D_k}.
\end{equation}
Let us show that $\a_k>0$ for all sufficiently large $k$. 
Since $b_{k-1}>0$ and $A_k\ge 0$ the desired inequality $\a_k>0$
can be rewritten as 
\begin{equation}
\frac{b_k}{b_{k-1}} < \frac{1-B_k}{1+A_k}.
\end{equation}
Set $M_r=\displaystyle{\sup_{x,k}}|f_k^{(r)}(x)|$
so that $A_k\le M_2(\frac{3t_k}{2}-t_k)$ and 
$B_k\le M_2(4t_k-\frac{8t_k}{3})$, which implies that
$\frac{1-B_k}{1+A_k}$ is bounded below by 
$\frac{1-4M_2t_k/3}{1+M_2 t_k/2}$
which tends to $1$ as $k\to\infty$. 
Since the sequence $(2^k b_k)$ is decreasing, we have 
$\frac{b_k}{b_{k-1}} < \frac{1}{2}$, and therefore
there exists $K$ such that for $\a_k>0$ for all $k \ge K$.

Now we are ready to analyze the behavior of $f$ in a
neighborhood of $0$. For $l\ge K$ let $Q_l$ be the partial
sum over $k\in [K, l]$ of the right hand side of (\ref{eq:1}). 
Let $q_{_l}$ be the unique solution of $f^{\prime\prime}=Q_l$, $f(0)=0=f^\prime(0)$;
note that $q_{_l}\in C^\infty(\R)$ and $q_{_l}(x)=0$ for all $x<C_l$ where
$C_l>0$ and $C_l\to 0$ as $l\to\infty$.
	
Since $A_k$ and $B_k$ tend to zero as $k\to\infty$, and $D_k$ decays like $t_k$, 
(\ref{eq: alpha_k}) implies 
that the sequence $(\a_k)$ super-exponentially converges to $0$. 
Together with uniform upper bound $|f_k^{(r)}|\le M_r$
this gives a uniform upper bounds on $|Q_l^{(r)}|$,
so by the Arzel{\`a}-Ascoli theorem the right hand side of (\ref{eq:1})
is $C^\infty$ near $0$, and hence so is the solution $f$ of (\ref{eq:1}).
Now it is straightforward to check that the 
solution of (\ref{eq:1}) is a pliable series at $0$; in particular,
$f^{(r)}(0)=0$ for each $r$.

Since $f_k^{\prime\prime}(x-t_k)>0$ for $x\in [0,2t_k]$,
on this interval 
we have $f_k^{\prime\prime}(x-t_k) \Psi_{2k}(x)>0$ if and only if $\Psi_{2k}(x)>0$.
Since $(0, 3t_K)$  is covered by the intervals where 
one of the functions $\Psi_{2k}$, $\Psi_{2k-1}$, $k\ge K$ is positive,
and since $\a_{k\ge K}$, $b_k$ are positive,
we conclude $f^{\prime\prime}>0$ on $(0, 3t_K]$, and integrating gives
$f^{\prime}>0$ on $(0, 3t_K]$.
\end{proof}

\bibliographystyle{amsalpha}
\bibliography{min}
\end{document}